\documentclass[11pt]{amsart}

\usepackage{amsmath,amssymb,mathtools}
\usepackage[T1]{fontenc}
\usepackage{lmodern}
\usepackage{microtype}
\usepackage{tikz-cd}
\usepackage[colorlinks=true,linkcolor=blue,citecolor=blue,urlcolor=blue]{hyperref}
\hypersetup{
 pdftitle={Real‑analytic realization of universal Teichm\"uller spaces
 },
 pdfauthor={Tailiang Liu and Yuliang Shen},
 pdfsubject={Universal Teichmuller space and operator coordinates},
 pdfkeywords={quasisymmetric homeomorphism, Grunsky operator, universal Teichm\"uller space}
}

\allowdisplaybreaks

\newtheorem{theorem}{Theorem}[section]
\newtheorem{proposition}[theorem]{Proposition}
\newtheorem{lemma}[theorem]{Lemma}
\newtheorem{corollary}[theorem]{Corollary}
\theoremstyle{remark}
\newtheorem{remark}[theorem]{Remark}

\newcommand{\D}{\mathbb D}
\newcommand{\Ds}{\mathbb D^{*}}
\newcommand{\C}{\mathbb C}
\newcommand{\R}{\mathbb R}

\newcommand{\T}{T(1)}
\newcommand{\QS}{\operatorname{QS}(S^{1})}
\newcommand{\Mob}{\operatorname{Mob}(S^{1})}
\newcommand{\Bop}{\mathrm B}
\newcommand{\HS}{\mathrm{HS}}
\newcommand{\id}{\mathrm{id}}
\newcommand{\ess}{\mathrm e}
\newcommand{\wh}{\widehat}

\title[complex-structure operators on $H^{1/2}$]
{Real-analytic realization of universal Teichm\"uller space via complex-structures  on $H^{1/2}$}
\author{Tailiang Liu}
\address{School of Mathematics and Physics, Jiangsu University of Technology, Changzhou 213001, P. R. China}
\email{ltlmath@163.com}
\author{Yuliang Shen}
\address{Department of Mathematics, Soochow University, Suzhou 215006, P. R. China}
\email{ylshen@suda.edu.cn}
\date{August 2026}

\begin{document}

\begin{abstract}
Let $H$ be the Hilbert transform, let $h$ be a quasisymmetric homeomorphism of the unit circle $S^1$, and
set $V_hu=u\circ h$ and $J_h=V_hHV_h^{-1}$, defined on the Sobolev space
$H^{\frac12}(S^1)$. We prove that $h\mapsto V_h$ is nowhere continuous in operator norm,
although it is  continuous  in the strong operator topology. By contrast, the
induced map $h\mapsto J_h$ is a real-analytic diffeomorphism onto its image
in the operator-norm topology.  Based on this, we further compute the differential at the identity and show that it is precisely the
Calder\'on commutator. In graph coordinates, the tangent map admits a weighted Hankel matrix representation, whose Hilbert–Schmidt norm recovers the Weil–Petersson tangent quadratic form.
\end{abstract}

\maketitle

\section{Introduction}

Let $\widehat{\C}$ be the Riemann sphere, $ \D=\{z\in\C:|z|<1\},
\Ds=\widehat{\C}\setminus\overline{\D},
S^1=\partial\D.$
A sense-preserving homeomorphism $F$ between planar domains is
\emph{quasiconformal} if $F\in W^{1,2}_{\mathrm{loc}}$ and its complex
dilatation $ \mu_F=\bar\partial F/\partial F $
satisfies $\|\mu_F\|_\infty<1$. Put
\[
 M(\D)=\{\mu\in L^\infty(\D):\|\mu\|_\infty<1\}.
\]
For $\mu\in M(\D)$, extend $\mu$ by zero to $\Ds$ and let $w^\mu$ be the
normalized solution of the Beltrami equation on $\widehat{\C}$ such that
\[
 w^\mu(z)=z+O(z^{-1})\qquad(z\to\infty).
\]
Two coefficients $\mu,\nu\in M(\D)$ are Teichm\"uller equivalent, denoted by $\mu\sim \nu$ if
$w^\mu=w^\nu$ on $\Ds$. The universal Teichm\"uller space is the quotient
\begin{equation}\label{eq:T-Beltrami-intro}
 \T=M(\D)/\sim,
\end{equation}
and $\pi:M(\D)\to\T$ denotes the quotient projection. We refer to
\cite{Ahlfors,GardinerLakic,Lehto,NagBook} for the standard theory.

The same space has a boundary-homeomorphism model. An orientation-preserving
homeomorphism $h:S^1\to S^1$ is \emph{quasisymmetric} if there is a constant
$M\geq1$ such that
\begin{equation}\label{eq:QS-intro}
 M^{-1}\leq\frac{|h(I_1)|}{|h(I_2)|}\leq M
\end{equation}
whenever $I_1$ and $I_2$ are adjacent arcs of equal length; here $|I|$
denotes arc length. By the Beurling--Ahlfors theorem, this is equivalent to
having a quasiconformal extension to $\D$; see
\cite{BeurlingAhlfors}. Let
$\QS$ be the group of such homeomorphisms and let $\Mob$ be the boundary
values of the conformal automorphisms of $\D$. Boundary correspondence gives
\begin{equation}\label{eq:T-QS-intro}
 \T=\QS\backslash\Mob.
\end{equation}
Thus $h$ and $m\circ h$ determine the same point when $m\in\Mob$. Every
class has a unique representative fixing $1$, $i$, and $-1$.

For $[h]\in\T$, let $h\in[h]$ be a representative. Let
\[
 f_h:\D\longrightarrow\Omega_h,
 \qquad
 g_h:\Ds\longrightarrow\Omega_h^*
\]
be a conformal welding pair satisfying
\begin{equation}\label{eq:welding}
 h=f_h^{-1}\circ g_h\quad\text{on }S^1,
 \qquad
 g_h(z)=z+O(z^{-1})\quad(z\to\infty).
\end{equation}
The exterior map $g_h$, which is independent of the choice of representative, is the restriction of $w^\mu$ to $\Ds$ for any
Beltrami representative $\mu$ of $[h]$.

The complex Banach-manifold structure of $\T$ can be described
by the Bers embedding. Let $B(\Ds)$ be the Banach space of holomorphic
functions $\phi$ on $\Ds$ with
\begin{equation}\label{eq:Bersnorm}
 \|\phi\|_B=\sup_{z\in\Ds}(|z|^2-1)^2|\phi(z)|<\infty.
\end{equation}
For a locally univalent holomorphic map $g$, write
\[
 S(g)=\frac{g'''}{g'}-\frac32\left(\frac{g''}{g'}\right)^2
\]
for its Schwarzian derivative. The map
$\mu\mapsto S(w^\mu|_{\Ds})$ is holomorphic on $M(\D)$ and constant on
Teichm\"uller classes. It descends to
\begin{equation}\label{eq:Bers}
 \beta:\T\longrightarrow B(\Ds),
 \qquad
 \beta([h])=S(g_h),
\end{equation}
which is a biholomorphism from $\T$ onto a bounded domain
$\Omega=\beta(\T)$ in $B(\Ds)$. Equivalently, $\pi$ is a holomorphic split
submersion and \eqref{eq:Bers} supplies the standard complex Banach-manifold
structure on $\T$; see \cite{NagBook}.

The corresponding Teichm\"uller distance can be written in the boundary
model as
\begin{equation}\label{eq:Teich-distance-intro}
 \tau([h_1],[h_2])=\frac12\inf_F\log K(F),
 \qquad
 K(F)=\frac{1+\|\mu_F\|_\infty}{1-\|\mu_F\|_\infty},
\end{equation}
where the infimum is taken over quasiconformal self-homeomorphisms of $\D$
whose boundary values are $h_2\circ h_1^{-1}$. We write $0=[\id]$ for the
base point of $\T$.

We now turn to the operator model studied in this paper. 
The space $H^{1/2}(S^{1})$ consists of equivalence classes (modulo constant functions) of complex‑valued functions $u:S^{1}\to\mathbb{C}$ such that
\begin{equation}
	\label{eq:Hhalf-intro}
	\|u\|_{H^{1/2}}^{2}:=\sum_{n\neq 0}|n|\,|\widehat{u}(n)|^{2}<\infty,\qquad 
	\widehat{u}(n)=\frac{1}{2\pi}\int_{0}^{2\pi}u(e^{i\theta})e^{-in\theta}\,d\theta.	
\end{equation}
Equivalently,
\[
\|u\|_{H^{1/2}}^{2}=\frac{1}{4\pi^{2}}\int_{S^{1}}\int_{S^{1}}\frac{|u(\zeta)-u(\eta)|^{2}}{|\zeta-\eta|^{2}}\,|d\zeta||d\eta|.
\]
If $Pu$ denotes the harmonic extension of $u$ to $\mathbb{D}$, then also
\[
\|u\|_{H^{1/2}}^{2}=\frac{1}{2\pi}\int_{\mathbb{D}}|\nabla Pu|^{2}\,dA.
\]
We write $H_{\mathbb{R}}^{1/2}(S^{1})$ for the closed subspace of real‑valued functions.
The Hilbert transform is the Fourier multiplier
\begin{equation}\label{eq:H-intro}
	H \Big(\sum_{n\neq0}u_ne^{in\theta}\Big)
	=-i\sum_{n\neq0}\operatorname{sgn}(n)u_ne^{in\theta},
\end{equation}
so that $H^2=-I$.  In particular, the Hilbert transform $H$ defines the standard
complex structure on $H_{\mathbb{R}}^{1/2}(S^{1})$.  Let
\[
H^{\frac12}_+(S^1)
=
\overline{\operatorname{span}}\{e^{in\theta}:n\ge1\},
\qquad
H^{\frac12}_-(S^1)
=
\overline{\operatorname{span}}\{e^{-in\theta}:n\ge1\},
\]
where the closures are taken in $H^{\frac12}(S^1)$. Equivalently, $H^{\frac12}_+(S^1)$ and $H^{\frac12}_-(S^1)$ are the
boundary trace spaces of the analytic and anti-analytic Dirichlet spaces
on $\D$, respectively. The latter may also be identified with the boundary
traces of holomorphic Dirichlet functions on $\Ds$ vanishing at $\infty$.  Then 
\[
H^{\frac12}(S^1)
=
H^{\frac12}_+(S^1)\oplus H^{\frac12}_-(S^1),
\]
and
\[
H=-iI\quad\text{on }H^{\frac12}_+(S^1),
\qquad
H=iI\quad\text{on }H^{\frac12}_-(S^1).
\]  Relative to this decomposition,
$H$ takes the diagonal block form
\[
H=\begin{pmatrix}
	-i\,I & 0\\
	0 & i\,I
\end{pmatrix}.
\]
The orthogonal projections onto $H^{\frac12}_+$ and $H^{\frac12}_-$ are then given by
\[
P_+=\frac12(I+iH),\qquad
P_-=\frac12(I-iH).
\] Each $h\in\QS$ induces the pull-back
\begin{equation}\label{eq:V-intro}
 V_hu=u\circ h.
\end{equation}
Clearly,
\[
 V_hV_k=V_{k\circ h},\qquad V_h^{-1}=V_{h^{-1}}.
\]
Nag and Sullivan proved that $V_h$ is a bounded automorphism of
$H^{\frac{1}{2}}(S^1)$ for quasisymmetric $h$, and they used this
action to construct the universal period mapping \cite{NagSullivan}.
 We recall from
\cite[Proposition~4.1]{ShenWP} that $h\mapsto V_h$ is continuous from the
Teichm\"uller metric to the strong operator topology; the operator-norm
continuity question was raised in \cite[Question~4.2]{ShenWP}.

For Banach spaces $E$ and $F$, write $\Bop(E,F)$ for the bounded linear
operators from $E$ to $F$ and $\Bop(E)=\Bop(E,E)$. The essential norm of
$T\in\Bop(E)$ is
\[
 \|T\|_{\ess}=\inf\{\|T-K\|:K\text{ is compact on }E\},
\]
and the minimum modulus of an invertible operator is
\[
 m(T)=\inf_{\|u\|=1}\|Tu\|=\|T^{-1}\|^{-1}.
\]
Our first result gives a quantitative negative answer to this question.

\begin{theorem}\label{thm:separation}
Let $h,k\in\QS$ be distinct, and regard $V_h$ and $V_k$ as operators on
$H^{\frac12}(S^1)$. Then
\begin{equation}\label{eq:essential-separation-intro}
 \|V_h-V_k\|_{\ess}
 \geq \bigl(m(V_h)^2+m(V_k)^2\bigr)^{1/2}
 =\bigl(\|V_h^{-1}\|^{-2}+\|V_k^{-1}\|^{-2}\bigr)^{1/2}.
\end{equation}
In particular, the family $\{V_h:h\in\QS\}$ is discrete in essential
norm, and hence also in operator norm.
\end{theorem}

\begin{corollary}\label{cor:nowhere}
Let $h_j,h$ be normalized quasisymmetric homeomorphisms and suppose that
$h_j\to h$ in the Teichm\"uller metric, with $h_j\neq h$. Then
\[
 \liminf_{j\to\infty}\|V_{h_j}-V_h\|_{\ess}>0.
\]
If $h=\id$, then
\[
 \liminf_{j\to\infty}\|V_{h_j}-I\|_{\ess}\geq\sqrt2.
\]
Nevertheless, $V_{h_j}u\to V_hu$ in $H^{\frac{1}{2}}(S^1)$ for every fixed
$u$.
\end{corollary}

A concrete smooth path tending to the identity is retained in
Proposition~\ref{prop:smoothpath}. It makes the failure of operator-norm
continuity visible without any abstract choice of a convergent sequence.
 
For $h\in\QS$, define
\begin{equation}\label{eq:Jdef-intro}
 J_h=V_hHV_h^{-1}.
\end{equation}
Then $J_h^2=-I$, so $J_h$ defines a complex structure on
$H^{\frac12}_{\R}(S^1)$. Its $-i$ and $i$ eigenspaces
are $V_h(H^{\frac{1}{2}}_{+})$ and $V_h(H^{\frac{1}{2}}_{-})$, respectively.
For $m\in\Mob$, we have $V_{m\circ h}=V_hV_m$ and
$J_{m\circ h}=J_h$. Thus $J_h$ depends only on the Teichm\"uller class
$[h]$.

\begin{theorem}
\label{thm:main-period}
The map
\[
 \Phi:\T\longrightarrow \mathcal{J}\subset \Bop(H^{\frac{1}{2}}(S^1)),
 \qquad \Phi([h])=J_h,
\]
is a real-analytic embedding in the operator-norm topology. Further, $\Phi^{-1}$ is the restriction of
a holomorphic map defined on an open neighborhood of $\mathcal{J}$ in
$\Bop(H^{\frac{1}{2}}(S^1))$. 
\end{theorem}
This local inverse has two immediate consequences.
\begin{corollary}\label{cor:topology}
For normalized $h_j,h\in\QS$,
\begin{equation}
\label{eq:top-equivalence}
[h_j]\longrightarrow[h]\text{ in }\T
\quad\Longleftrightarrow\quad
\|J_{h_j}-J_h\|\longrightarrow0.
\end{equation}
 For every $[h_0]\in\T$ there are a neighborhood $U$ of $[h_0]$ and a
constant $C\geq1$ such that, for $[h],[k]\in U$,
\begin{equation}
	\label{eq:localcomparison}
	C^{-1}\|\beta([h])-\beta([k])\|_B
	\leq\|J_h-J_k\|
	\leq C\|\beta([h])-\beta([k])\|_B.
\end{equation}
\end{corollary}
Recall that $h\in\QS$ is called \emph{symmetric} if
\[
\lim_{\delta\to 0}
\sup_{\substack{|I_1|=|I_2|\le\delta}}
\frac{|h(I_1)|}{|h(I_2)|}=1,
\]
where $I_1$ and $I_2$ are adjacent arcs of $S^1$, and that $h$ belongs to
the Weil--Petersson class $\operatorname{WP}(S^1)$ if it has a
quasiconformal extension $F$ to $\D$ satisfying
\[
 \int_{\D}\frac{|\mu_F(z)|^2}{(1-|z|^2)^2}\,dA(z)<\infty.
\]
 The results of \cite{HuShen} show that the anti-analytic pull-back is
 compact exactly on the symmetric  class and Hilbert--Schmidt exactly on the
 Weil--Petersson class. In the present coordinate this becomes the following. 

\begin{corollary}
\label{cor:hushen}
Let $h\in\QS$. Then
\[
 h\text{ is symmetric}
 \quad\Longleftrightarrow\quad
 J_h-H\text{ is compact},
\]
and
\[
 h\in\operatorname{WP}(S^1)
 \quad\Longleftrightarrow\quad
 J_h-H\text{ is Hilbert--Schmidt}.
\]
\end{corollary}

There is an earlier BMO counterpart of the forward real-analyticity in Theorem \ref{thm:main-period}. Recall that  a strongly quasisymmetric homeomorphism $h$ on the real  line $\mathbb{R}$ is defined to be a locally absolutely continuous  with $h'$ is an $A_\infty$ weight.  In particular, \(\log h'\in \mathrm{BMO}\). Coifman and Meyer  \cite{CoifmanMeyer} proved that the map
\[
\log h'\longmapsto V_hHV_h^{-1}
\]
is  real-analytic 
 by real‑variable methods; see also \cite{CoifmanSemmes,SemmesCauchy}. In addition, they also noted the map  $h\to V_h$ fails to be continuous in the sense of  the operator norm, without giving a proof; see \cite[corollary 8.6]{WeiMatsuzaki} for more details.

The paper is organized as follows. Section~\ref{sec:pullbacks} proves the
essential-norm separation theorem and constructs an explicit smooth path demonstrating the absence of operator‑norm convergence.
Section~\ref{sec:period} serves as the main body of this paper, devoted to proving that the embedding \([h]\mapsto J_h\) is a real-analytic diffeomorphism onto its image. Section~\ref{sec:infinitesimal} computes the
differential in boundary coordinates and recovers the M\"obius kernel and
the Weil--Petersson form.

\section{Pull-backs on the critical Sobolev space}\label{sec:pullbacks}

\subsection{The pull-back action}
For smooth real-valued functions modulo constants, put
\begin{equation}
 \omega(u,v)=\frac1{2\pi}\int_{S^{1}}u\,dv.
\end{equation}
A direct Fourier calculation shows that it extends continuously to a nondegenerate symplectic form on
$H^{\frac{1}{2}}_{\R}(S^{1})$, and $H$ is compatible with it. This is the symplectic
Hilbert space used by Nag and Sullivan \cite{NagSullivan}.

We use the same notation $V_h$ for the complex-linear extension to
$H^{\frac{1}{2}}(S^{1})$. Complexification preserves its operator norm. 

\begin{theorem}[\cite{NagSullivan}]\label{thm:NS-boundedness}
For an orientation-preserving homeomorphism $h:S^{1}\to S^{1}$, the operator
$V_h$ is a bounded automorphism of $H^{\frac{1}{2}}_{\R}(S^{1})$ if and only if
$h\in\QS$. In that case
\[
 V_h^{-1}=V_{h^{-1}},
\]
$V_h$ preserves the symplectic form, and
\begin{equation}\label{eq:Vnorm}
 \|V_h\|\leq e^{\tau(0,[h])}.
\end{equation}
\end{theorem}
The norm estimate follows from
quasiconformal quasi-invariance of the Dirichlet integral . Inverting a quasiconformal extension
does not change its maximal dilatation, so
$\tau(0,[h^{-1}])=\tau(0,[h])$. Applying \eqref{eq:Vnorm} to $h^{-1}$ therefore
gives
\begin{equation}\label{eq:inverse-minimum-bound}
 \|V_h^{-1}\|\leq e^{\tau(0,[h])},\qquad
 m(V_h)\geq e^{-\tau(0,[h])}.
\end{equation}

\begin{theorem}[Strong-operator continuity]\label{thm:strong-cont}
Let $h_j$ and $h$ be normalized quasisymmetric homeomorphisms. If
$h_j\to h$ in the Teichm\"uller topology, then
\[
 V_{h_j}u\longrightarrow V_hu\quad\text{in }H^{\frac{1}{2}}(S^{1})
\]
for every fixed $u\in H^{\frac{1}{2}}(S^{1})$.
\end{theorem}

This is \cite[Proposition~4.1]{ShenWP}. The distinction between
strong-operator continuity and operator-norm continuity is the starting point
of the present paper.

\subsection{Concentrating packets}

Fix a nonzero, nonconstant real-valued function
$\eta\in C_c^{\infty}((-1,1))$. For $\zeta_0=e^{i\theta_0}$ and sufficiently
small $\varepsilon>0$, define
\begin{equation}
 u_{\varepsilon,\zeta_0}(e^{i\theta})
 =\eta\!\left(\frac{\theta-\theta_0}{\varepsilon}\right),
\end{equation}
using a fixed coordinate arc around $\theta_0$.

\begin{lemma}[Critical scaling]\label{lem:critical-scaling}
There is a number $a_\eta\in(0,\infty)$ such that
\begin{equation}
 \|u_{\varepsilon,\zeta_0}\|_{H^{\frac{1}{2}}}^2\longrightarrow a_\eta
 \qquad(\varepsilon\downarrow0).
\end{equation}
Moreover,
\begin{equation}
 \|u_{\varepsilon,\zeta_0}\|_{L^1(S^{1})}
 =\varepsilon\|\eta\|_{L^1(\R)},
 \qquad
 \|u_{\varepsilon,\zeta_0}\|_{L^{\infty}(S^{1})}
 =\|\eta\|_{L^{\infty}(\R)}.
\end{equation}
\end{lemma}

\begin{proof}
By rotation invariance, take $\theta_0=0$. Let
\[
 \wh\eta_{\R}(\xi)=\int_{\R}\eta(x)e^{-i\xi x}\,dx.
\]
For sufficiently small $\varepsilon$,
\[
 \wh u_{\varepsilon,1}(n)=\frac{\varepsilon}{2\pi}
 \wh\eta_{\R}(n\varepsilon).
\]
Consequently,
\begin{align*}
 \|u_{\varepsilon,1}\|_{H^{\frac{1}{2}}}^2
 &=\frac{\varepsilon^2}{4\pi^2}
   \sum_{n\neq0}|n|\,|\wh\eta_{\R}(n\varepsilon)|^2\\
 &=\frac1{4\pi^2}\sum_{n\neq0}
   \varepsilon |n\varepsilon|\,|\wh\eta_{\R}(n\varepsilon)|^2.
\end{align*}
The last expression is a Riemann sum for
\[
 a_\eta=\frac1{4\pi^2}\int_{\R}|\xi|\,|\wh\eta_{\R}(\xi)|^2\,d\xi.
\]
It is positive because $\eta$ is nonconstant. The remaining identities are
immediate.
\end{proof}

Set
\begin{equation}
 v_{\varepsilon,\zeta_0}
 =\frac{u_{\varepsilon,\zeta_0}}
 {\|u_{\varepsilon,\zeta_0}\|_{H^{\frac{1}{2}}}}.
\end{equation}
Then $\|v_{\varepsilon,\zeta_0}\|_{H^{\frac{1}{2}}}=1$, their $L^1$ norms tend to
zero, and their $L^{\infty}$ norms remain bounded.

\begin{lemma}[Weak convergence]\label{lem:weak}
For each fixed $\zeta_0\in S^{1}$,
\[
 v_{\varepsilon,\zeta_0}\rightharpoonup0
 \quad\text{weakly in }H^{\frac{1}{2}}(S^{1})
\]
as $\varepsilon\downarrow0$.
\end{lemma}

\begin{proof}
It is enough to test against trigonometric polynomials modulo constants. For
each fixed $n\neq0$,
\[
 \wh u_{\varepsilon,\zeta_0}(n)=O(\varepsilon).
\]
The denominator in the definition of $v_{\varepsilon,\zeta_0}$ converges to
a positive number. Hence every fixed nonzero Fourier coefficient tends to
zero. The family is bounded in $H^{\frac{1}{2}}(S^{1})$, and density of complex
trigonometric polynomials finishes the proof.
\end{proof}

\begin{lemma}[Separated supports]\label{lem:separated-supports}
Let $E,F\subset S^{1}$ satisfy $\operatorname{dist}(E,F)\geq d>0$. If
$u,v\in H^{\frac{1}{2}}_{\R}(S^{1})\cap L^1(S^{1})$ have representatives supported
in $E$ and $F$, respectively, then
\begin{equation}
 |\langle u,v\rangle_{H^{\frac{1}{2}}}|
 \leq c_0d^{-2}\|u\|_{L^1}\|v\|_{L^1}.
\end{equation}
\end{lemma}

\begin{proof}
The integrand is nonzero only when one variable lies
in $E$ and the other lies in $F$. After interchanging the variables in one of
the two contributions,
\[
 \langle u,v\rangle_{H^{\frac{1}{2}}}
 =\frac{1}{4\pi^2}\int_E\int_F\frac{u(\zeta)v(\eta)}{|\zeta-\eta|^2}
 \,|d\zeta|\,|d\eta|.
\]
Estimate the $|\zeta-\eta|$ from below by $d^2$.
\end{proof}

\subsection{Essential-norm separation}

\begin{proof}[Proof of Theorem~\ref{thm:separation}]
Choose $\zeta_0\in S^{1}$ such that
\[
 x_h=h^{-1}(\zeta_0)\neq k^{-1}(\zeta_0)=x_k.
\]
Let $v_\varepsilon=v_{\varepsilon,\zeta_0}$. The support of $V_hv_\varepsilon$
is $h^{-1}(\operatorname{supp}v_\varepsilon)$, which shrinks to $x_h$.
Similarly, the support of $V_kv_\varepsilon$ shrinks to $x_k$. For
sufficiently small $\varepsilon$, these two supports are separated by a fixed
positive distance.

The functions $v_\varepsilon$ have uniformly bounded $L^{\infty}$ norms.
Since the lengths of both support arcs tend to zero,
\[
 \|V_hv_\varepsilon\|_{L^1}\to0,
 \qquad
 \|V_kv_\varepsilon\|_{L^1}\to0.
\]
By Lemma~\ref{lem:separated-supports},
\begin{equation}\label{eq:orthogonal}
 \langle V_hv_\varepsilon,V_kv_\varepsilon\rangle_{H^{\frac{1}{2}}}\to0.
\end{equation}
On the other hand,
\[
 \|V_hv_\varepsilon\|_{H^{\frac{1}{2}}}\geq m(V_h),
 \qquad
 \|V_kv_\varepsilon\|_{H^{\frac{1}{2}}}\geq m(V_k).
\]
Therefore
\begin{align*}
 \|(V_h-V_k)v_\varepsilon\|_{H^{\frac{1}{2}}}^2
 &=\|V_hv_\varepsilon\|_{H^{\frac{1}{2}}}^2
   +\|V_kv_\varepsilon\|_{H^{\frac{1}{2}}}^2
   -2\operatorname{Re}\langle V_hv_\varepsilon,V_kv_\varepsilon\rangle_{H^{\frac{1}{2}}}\\
 &\geq m(V_h)^2+m(V_k)^2-o(1).
\end{align*}
This proves the operator-norm lower bound.

Let $K$ be compact on $H^{\frac{1}{2}}(S^{1})$. By
Lemma~\ref{lem:weak}, $Kv_\varepsilon\to0$ in norm. Hence
\[
 \|(V_h-V_k)-K\|
 \geq\liminf_{\varepsilon\downarrow0}
 \|(V_h-V_k-K)v_\varepsilon\|
 \geq\bigl(m(V_h)^2+m(V_k)^2\bigr)^{1/2}.
\]
Taking the infimum over compact $K$ proves the theorem.
\end{proof}

\begin{proof}[Proof of Corollary~\ref{cor:nowhere}]
Since $h_j\to h$ in the Teichm\"uller topology, the quantities
$\tau(0,[h_j])$ remain bounded. By \eqref{eq:inverse-minimum-bound}, the norms
$\|V_{h_j}^{-1}\|$ remain bounded, and therefore $m(V_{h_j})$ is bounded
below by a positive constant. Apply Theorem~\ref{thm:separation} to $h_j$
and $h$.

If $h=\id$, then
\[
 m(V_{h_j})=\|V_{h_j}^{-1}\|^{-1}
 \geq e^{-\tau(0,[h_j])}\longrightarrow1,
\]
which gives the lower bound $\sqrt2$. The final statement is
Theorem~\ref{thm:strong-cont}.
\end{proof}

\subsection{A smooth path}

For completeness, we record an explicit nonconstant smooth path to which
the preceding corollary applies.

\begin{proposition}\label{prop:smoothpath}
For $|t|<1/4$, define
\[
 h_t(e^{i\theta})=e^{i(\theta+t\sin2\theta)}.
\]
Then $h_t$ is a normalized quasisymmetric homeomorphism,
$h_t\to\id$ in the Teichm\"uller metric, and
\[
 \liminf_{t\to0,\,t\neq0}\|V_{h_t}-I\|_{\ess}\geq\sqrt2.
\]
\end{proposition}

\begin{proof}
The angular derivative is $1+2t\cos2\theta>0$, and the points $1,i,-1$ are
fixed. Define
\[
 F_t(re^{i\theta})=r\exp\!\bigl(i(\theta+tr^2\sin2\theta)\bigr).
\]
A direct calculation gives
\begin{align*}
 (F_t)_z&=e^{i(\Phi_t-\theta)}
 \bigl(1+tr^2\cos2\theta+itr^2\sin2\theta\bigr),\\
 (F_t)_{\bar z}&=tr^2e^{i(\Phi_t+\theta)}
 \bigl(-\cos2\theta+i\sin2\theta\bigr),
\end{align*}
where $\Phi_t=\theta+tr^2\sin2\theta$. Thus
\[
 \|\mu_{F_t}\|_{\infty}\leq\frac{|t|}{1-|t|}\longrightarrow0.
\]
Hence $h_t\to\id$ in the Teichm\"uller metric. The essential-norm conclusion
follows from Corollary~\ref{cor:nowhere}.
\end{proof}

\section{The universal period map through Grunsky coordinates}\label{sec:period}

\subsection{$V_h$ and its graph}

A bounded real-linear operator $J$ on $H^{\frac12}_{\R}(S^1)$ is called a
 complex structure compatible with the symplectic form $\omega$ if
\[
 J^2=-I,\qquad \omega(Ju,Jv)=\omega(u,v),\qquad
 \omega(u,Ju)>0\quad(u\ne0);
\] This definition is given in Nag--Sullivan~\cite[Section~7]{NagSullivan}, though we do not make essential use of it in the sequel.
Because $V_h$ is symplectic  operator on $H^{\frac{1}{2}}_{\R}(S^{1})$, the operator
\begin{equation}\label{eq:Jdef}
 J_h=V_hHV_h^{-1}
\end{equation}
is again a  complex structure compatible with $\omega$.
Its eigenspaces have a simple description.

\begin{lemma}\label{lem:eigenspaces}
The $-i$ eigenspace of $J_h$ is $V_h(H^{\frac{1}{2}}_{+})$, and the $i$ eigenspace is
$V_h(H^{\frac{1}{2}}_{-})$.
\end{lemma}

\begin{proof}
If $f\in H^{\frac{1}{2}}_{+}$, then $Hf=-if$, and hence
\[
 J_h(V_hf)=V_hHf=-iV_hf.
\]
The second assertion follows in the same way, and invertibility of $V_h$
gives the whole eigenspaces.
\end{proof}

We also introduce here the graph coordinate used in the statements below. For $h\in\QS$,  define
\begin{equation}\label{eq:ABdef}
 A_h=P_+V_h|_{H^{\frac{1}{2}}_{+}},\qquad
 B_h=P_-V_h|_{H^{\frac{1}{2}}_{+}}.
\end{equation}

\begin{theorem}[\cite{ShenWP}, Proposition~10.2]\label{thm:Ainvertible}
For every $h\in\QS$, the operator
\[
 A_h:H^{\frac{1}{2}}_{+}\longrightarrow H^{\frac{1}{2}}_{+}
\]
is a bounded isomorphism.
\end{theorem}
 We therefore define the graph
operator
\begin{equation}\label{eq:Zdef}
 Z_h=B_hA_h^{-1}:H^{\frac{1}{2}}_{+}\longrightarrow H^{\frac{1}{2}}_{-}.
\end{equation}
If $u=A_hf$, then $V_hf=u+Z_hu$, and hence
\begin{equation}\label{eq:graph}
 V_h(H^{\frac{1}{2}}_{+})=\{u+Z_hu:u\in H^{\frac{1}{2}}_{+}\}.
\end{equation}
It is easy to see that $Z_h$ also depends only on $[h]$. 
Thus $[h]\mapsto Z_h$ is well defined on $\T$.
 
 The Grunsky operator played an important role in the study of univalent function theory
and yielded a lot of properties of Teichm\"uller spaces as well.
We now briefly recall some facts concerning the Grunsky operator; see \cite{Pommerenke1975}. For detailed applications and its connection with Faber polynomials, we refer to our earlier paper \cite{LiuShenFaber}. 

We use the welding pair $(f_h,g_h)$ fixed in \eqref{eq:welding}; in
particular, $f_h\circ h=g_h$ on $S^{1}$, and the normalized exterior map
$g_h$ depends only on $[h]$.

The Grunsky coefficients $\alpha_{mn}(g_h)$ are determined by
\begin{equation}\label{eq:grunskyexp}
 \log\frac{g_h(z)-g_h(\zeta)}{z-\zeta}
 =-\sum_{m,n\geq1}\alpha_{mn}(g_h)z^{-m}\zeta^{-n}.
\end{equation}
Define the Grunsky operator $G([h])$
by
\[
\bigl(G([h])x\bigr)_m
=\sum_{n=1}^\infty
\sqrt{mn}\,\alpha_{mn}(g_h)x_n,
\qquad m\ge1,
\]
for $x=(x_n)_{n\ge1}\in\ell^2$.
 Its matrix representation is
\begin{equation}\label{eq:Gdef}
 G([h])=\bigl(\sqrt{mn}\,\alpha_{mn}(g_h)\bigr)_{m,n\geq1}.
\end{equation}

The Grunsky inequalities and quasiconformal extendibility give  $G([h])$ is bounded on $\ell^2$ and
$\|G([h])\|<1$.

Let $a=(a_n)_{n\geq1}\in \ell^2$ and
\begin{equation}\label{eq:Fourierunitaries}
 F_+(a)=\sum_{n\geq1}a_n e_n^+,
 \qquad
 F_-(a)=\sum_{n\geq1}a_n e_n^-,
\end{equation}
where
\[
 e_n^+(\theta)=\frac{e^{in\theta}}{\sqrt n},
 \qquad
 e_n^-(\theta)=\frac{e^{-in\theta}}{\sqrt n}.
\]
These are unitary maps from $\ell^2$ onto $H^{\frac{1}{2}}_{+}$ and
$H^{\frac{1}{2}}_{-}$, respectively.

\begin{proposition}[Faber--Grunsky graph identity]\label{prop:graph-grunsky}
For every $[h]\in\T$,
\begin{equation}\label{eq:graph-grunsky}
 F_-^{-1}Z_hF_+=G([h]).
\end{equation}
Equivalently, for every $n\geq1$,
\begin{equation}\label{eq:Zbasis}
 Z_he_n^+=\sum_{m\geq1}\sqrt{mn}\,\alpha_{mn}(g_h)e_m^-.
\end{equation}
\end{proposition}

\begin{proof}
	Let $P_n$ be the $n$th Faber polynomial of the exterior map $g_h$, normalized by
	the following expression:
	\begin{equation}\label{eq:Faber-generating}
		\log\frac{g_h(z)-w}{z}
		=-\sum_{n\geq1}\frac1nP_n(w)z^{-n},
		\qquad(w\in\mathbb{C}, |z|\text{ large}).
	\end{equation}
Putting $w=g_h(\zeta)$ and using \eqref{eq:grunskyexp}, we obtain
\begin{align*}
 -\sum_{n\geq1}\frac1nP_n(g_h(\zeta))z^{-n}
 &=\log\left(1-\frac{\zeta}{z}\right)
   +\log\frac{g_h(z)-g_h(\zeta)}{z-\zeta}\\
 &=-\sum_{n\geq1}\frac{\zeta^n}{n}z^{-n}
   -\sum_{m,n\geq1}\alpha_{mn}(g_h)z^{-m}\zeta^{-n}.
\end{align*}
Comparing the coefficient of $z^{-n}$ gives the Faber expansion
\begin{equation}\label{eq:Faber-expansion}
 P_n(g_h(\zeta))
 =\zeta^n+n\sum_{m\geq1}\alpha_{nm}(g_h)\zeta^{-m},
 \qquad |\zeta|>1.
\end{equation}

Since $\Omega_h$ is bounded,  conformal invariance of the Dirichlet integral gives
\[
 \int_{\D}|(P_n\circ f_h(z))'|^2\,dm(z)
 =\int_{\Omega_h}|P_n'(w)|^2\,dm(w)<\infty.
\]
Hence $u_n(z)=n^{-1/2}P_n\circ f_h(z)$ has boundary trace in
$H^{\frac{1}{2}}_{+}$. Moreover,
$P_n\circ g_h=P_n\circ f_h\circ h=\sqrt n\,V_hu_n$ on $S^1$, so this
continuous boundary trace belongs to $H^{\frac{1}{2}}$. By Cauchy's formula on
$|\zeta|=r$ and passage to $r\downarrow1$, the Laurent coefficients in
\eqref{eq:Faber-expansion} are its Fourier coefficients; since the trace is
in $H^{\frac{1}{2}}$, the resulting Fourier series converges in $H^{\frac{1}{2}}$. Therefore,
\begin{align*}
 V_hu_n
 &=\frac1{\sqrt n}P_n\circ f_h\circ h
  =\frac1{\sqrt n}P_n\circ g_h\\
 &=e_n^++\sum_{m\geq1}\sqrt{mn}\,
   \alpha_{nm}(g_h)e_m^-.
\end{align*}
The Grunsky symmetry $\alpha_{nm}=\alpha_{mn}$ then gives
\[
 A_hu_n=e_n^+,
 \qquad
 B_hu_n=\sum_{m\geq1}\sqrt{mn}\,\alpha_{mn}(g_h)e_m^-.
\]
Since $A_h$ is invertible, $u_n=A_h^{-1}e_n^+$, and
\eqref{eq:Zbasis} follows. Equality on the orthonormal basis gives
\eqref{eq:graph-grunsky}.
\end{proof}

\begin{corollary}[Holomorphy of the graph coordinate]\label{cor:Zholomorphic}
The map
\[
 Z:\T\longrightarrow\Bop(H^{\frac{1}{2}}_{+},H^{\frac{1}{2}}_{-}),
 \qquad [h]\longmapsto Z_h,
\]
is holomorphic in operator norm.
\end{corollary}

\begin{proof}   Now
Proposition~\ref{prop:graph-grunsky} gives
\[
 Z_h=F_-G([h])F_+^{-1},
\]
and the Fourier maps are fixed and complex linear. It is known that the Grunsky operator depends
holomorphically on each point $h$ in the universal Teichm\"uller space; see \cite{Krushkal1989,ShenGrunsky,TakhtajanTeo}. Then $Z_h=B_hA_h^{-1}$ is holomorphic. 
\end{proof}
\begin{remark}
	In fact, Nag and Sullivan have proved that $ [h]\longmapsto Z_h $ is an injective, equivariant,
	holomorphic immersion
	\cite[Theorem~7.1]{NagSullivan} by   Rauch variational
	argument. 
	Here it is obtained through the
	Faber--Grunsky identity and the known operator-norm
	holomorphy of the Grunsky map. 
\end{remark}
\subsection{Bers--Grunsky biholomorphism}

The works of Takhtajan--Teo \cite[Appendix B]{TakhtajanTeo} and Krushkal \cite{Krushkal1989,Krushkal2017} already imply a biholomorphic Grunsky realization of the universal Teichm\"uller space. We briefly recall this in our normalization, in a form that will be used later for the complex-structure coordinate.

Recall from \eqref{eq:Bers} that $\beta([h])=S(g_h)$ identifies $\T$
biholomorphically with $\Omega\subset B(\Ds)$.
For $z\in\Ds$, put
\begin{equation}\label{eq:vz}
 v(z)=(\sqrt n\,z^{-n})_{n\geq1}\in\ell^2,
 \qquad
 [x,y]_0=\sum_{n\geq1}x_ny_n.
\end{equation}
The pairing $[\cdot,\cdot]_0$ is complex bilinear and continuous on
$\ell^2\times\ell^2$. For $T\in\Bop(\ell^2)$ define
\begin{equation}\label{eq:Edef}
 E(T)(z)=-6z^{-2}[Tv(z),v(z)]_0.
\end{equation}

\begin{lemma}\label{lem:E}
Formula \eqref{eq:Edef} defines a bounded complex-linear map
\[
 E:\Bop(\ell^2)\longrightarrow B(\Ds)
\]
satisfying
\begin{equation}\label{eq:Ebound}
 \|E(T)\|_B\leq6\|T\|.
\end{equation}
\end{lemma}

\begin{proof}
Since
\[
 \|v(z)\|_{\ell^2}^2
 =\sum_{n=1}^{\infty}n|z|^{-2n}
 =\frac{|z|^2}{(|z|^2-1)^2},
\]
the Cauchy--Schwarz inequality gives
\[
 |E(T)(z)|
 \leq6|z|^{-2}\|T\|\|v(z)\|_{\ell^2}^2
 =\frac{6\|T\|}{(|z|^2-1)^2}.
\]
The map $z\mapsto v(z)$ is holomorphic as an $\ell^2$-valued map on $\Ds$,
so $E(T)$ is holomorphic. This proves \eqref{eq:Ebound} and the
asserted bounded complex linearity.
\end{proof}

\begin{proposition}[Bers--Grunsky recovery]\label{prop:recovery}
For every $[h]\in\T$,
\begin{equation}\label{eq:betaEG}
 \beta([h])=E(G([h])).
\end{equation}
Consequently,
\begin{equation}\label{eq:Bers-Grunsky-bound}
 \|\beta([h])-\beta([k])\|_B
 \leq6\|G([h])-G([k])\|.
\end{equation}
\end{proposition}

\begin{proof}
 Differentiating \eqref{eq:grunskyexp} once in each variable
and then taking the removable diagonal limit $\zeta\to z$ gives
\begin{equation}\label{eq:Schwarzian-diagonal}
 S(g_h)(z)=-6\sum_{m,n\geq1}mn\alpha_{mn}(g_h)z^{-(m+n+2)}.
\end{equation}
By direct computation,
\begin{align*}
	E(G([h]))(z)
	&=-6z^{-2}\big[G([h])v(z),v(z)\big]_0\\
	&=-6z^{-2}\sum_{m,n\ge 1} mn\,\alpha_{mn}(g_h)\,z^{-(m+n)}\\
	&=-6\sum_{m,n\ge 1} mn\,\alpha_{mn}(g_h)\,z^{-(m+n+2)}\\
	&=S(g_h)(z)=\beta([h])(z).
\end{align*}
On the other hand, the estimate
\eqref{eq:Bers-Grunsky-bound} follows from Lemma~\ref{lem:E}.
\end{proof}

\begin{corollary}[The inverse of the Grunsky map]\label{cor:Ginverse}
The map
\[
 G:\T\longrightarrow G(\T)\subset\Bop(\ell^2)
\]
is  biholomorphism. Its inverse is
\begin{equation}\label{eq:Ginverse}
 G^{-1}=\beta^{-1}\circ E
 \qquad\text{on }G(\T).
\end{equation}
Moreover, the right-hand side is holomorphic on the open set
$E^{-1}(\Omega)\subset\Bop(\ell^2)$, which contains the Grunsky class.
\end{corollary}

\begin{proof}
Formula \eqref{eq:Ginverse} follows from \eqref{eq:betaEG} and
injectivity of the Bers embedding. Since $\Omega$ is open and $E$ is
bounded linear, $E^{-1}(\Omega)$ is open, and
$\beta^{-1}\circ E$ is holomorphic there.
\end{proof}

\subsection{Graph coordinates}

For two
independent operators, 
\[
L:H^{\frac{1}{2}}_{+}\to H^{\frac{1}{2}}_{-},
\qquad
W:H^{\frac{1}{2}}_{-}\to H^{\frac{1}{2}}_{+},
\]
we define $O$ as the open set of pairs $_{L,W}$ for which $I-WL$ and $I-LW$ are invertible.
 Set
\begin{equation}\label{eq:RLW}
 R_{L,W}=\begin{pmatrix}I&W\\ L&I\end{pmatrix}.
\end{equation}
Then for $(L,W)\in O$,
\begin{equation}\label{eq:Rinverse}
 R_{L,W}^{-1}
 =\begin{pmatrix}
 (I-WL)^{-1}&-W(I-LW)^{-1}\\
 -L(I-WL)^{-1}&(I-LW)^{-1}
 \end{pmatrix}.
\end{equation}
Let $C$ denote complex conjugation on $H^{\frac{1}{2}}$; it interchanges
$H^{\frac{1}{2}}_{+}$ and $H^{\frac{1}{2}}_{-}$. Thus, if
$Z:H^{\frac{1}{2}}_{+}\to H^{\frac{1}{2}}_{-}$ is complex linear, then
$\widetilde{Z}\doteq CZC:H^{\frac{1}{2}}_{-}\to H^{\frac{1}{2}}_{+}$ is also complex linear. 
Define the open graph domain
\[
 \mathbb{G}=\left\{Z\in\Bop(H^{\frac12}_+,H^{\frac12}_-):
 (Z,\widetilde{Z})\in O\right\}.
\]
It contains every $Z_h$. Indeed, Proposition~\ref{prop:graph-grunsky} and
the Grunsky inequalities give $\|Z_h\|<1$; since $C$ is isometric, both
$I-\widetilde{Z_h}Z_h$ and $I-Z_h\widetilde{Z_h}$ are  invertible by the Neumann
series. For $Z\in\mathbb{G}$, define
\begin{equation}\label{eq:JofZ}
 R_Z=\begin{pmatrix}I&\widetilde{Z}\\ Z&I\end{pmatrix},\qquad J(Z)=R_ZHR_Z^{-1}.
\end{equation}

\begin{proposition}[Graph and complex-structure coordinates]
\label{prop:graph-J}
The map $J:\mathbb{G}\to\Bop(H^{\frac12})$, $Z\mapsto J(Z)$, is real
analytic in operator norm. Its $-i$
eigenspace is the graph $\{u+Zu:u\in H^{\frac{1}{2}}_{+}\}$. 
With respect to the decomposition
$H^{\frac{1}{2}}=H^{\frac{1}{2}}_{+}\oplus H^{\frac{1}{2}}_{-}$,
the operator $J(Z)$ admits the block decomposition
\begin{equation}\label{eq:Jblock}
	J(Z)=
	\begin{pmatrix}
		-i(I+\widetilde{Z}Z)(I-\widetilde{Z}Z)^{-1} & 2i\widetilde{Z}(I-Z\widetilde{Z})^{-1}\\[6pt]
		-2iZ(I-\widetilde{Z}Z)^{-1} & i(I+Z\widetilde{Z})(I-Z\widetilde{Z})^{-1}
	\end{pmatrix}.
\end{equation}

Conversely, let $J$ satisfy that
$P_+Q_J|_{H^{1/2}_{+}}$ is invertible, where
\begin{equation}\label{eq:QJ}
 Q_J=\frac12(I+iJ).
\end{equation}
Define
\begin{equation}\label{eq:ZofJ}
 \mathcal{Z}(J)=P_-Q_J|_{H^{1/2}_{+}}
 \bigl(P_+Q_J|_{H^{1/2}_{+}}\bigr)^{-1}.
\end{equation}
Then $J\mapsto \mathcal{Z}(J)$ is holomorphic on this open set, and
\begin{equation}\label{eq:inverse-graph}
 \mathcal{Z}(J(Z))=Z.
\end{equation}
\end{proposition}

\begin{proof}
The map $(L,W)\mapsto R_{L,W}HR_{L,W}^{-1}$ is holomorphic on $O$ by
\eqref{eq:Rinverse}. Since $Z\mapsto(Z,\widetilde{Z})$ is bounded real linear,
$Z\mapsto J(Z)$ is real analytic. Multiplication of the block matrices gives
\eqref{eq:Jblock}.

For $u\in H^{\frac{1}{2}}_{+}$,
\[
 R_Zu=u+Zu,
 \qquad
 J(Z)(u+Zu)=R_ZHu=-i(u+Zu),
\]
so the $-i$ eigenspace is the stated graph.

The map \(J\mapsto Q_J = \tfrac12(I+iJ)\) is complex affine, hence holomorphic. Moreover, the operator inversion map is holomorphic on the open set of invertible bounded operators, so \(P_+Q_J|_{H^{1/2}_+}^{-1}\) depends holomorphically on \(J\) whenever \(P_+Q_J|_{H^{1/2}_+}\) is invertible. As a product of holomorphic operator-valued maps, \(\mathcal{Z}(J) = \bigl(P_-Q_J|_{H^{1/2}_+}\bigr)\bigl(P_+Q_J|_{H^{1/2}_+}\bigr)^{-1}\) is therefore holomorphic.

A direct conjugation yields
\[
Q_{J(Z)}=\tfrac12\big(I+iJ(Z)\big) = R_Z\cdot \tfrac12(I+iH)\cdot R_Z^{-1} = R_Z P_+ R_Z^{-1}.
\]
Restricting this operator to \(H^{1/2}_+\), we compute its action on any \(u\in H^{1/2}_+\). Recall that
\[
R_Z=\begin{pmatrix}I & \widetilde Z\\ Z & I\end{pmatrix},\qquad
R_Z^{-1}=\begin{pmatrix}(I-\widetilde Z Z)^{-1} & -\widetilde Z(I-Z\widetilde Z)^{-1}\\ -Z(I-\widetilde Z Z)^{-1} & (I-Z\widetilde Z)^{-1}\end{pmatrix},
\]
we obtain
\[
Q_{J(Z)}u = R_Z P_+ R_Z^{-1}\begin{pmatrix}u\\0\end{pmatrix}
= R_Z\begin{pmatrix}(I-\widetilde Z Z)^{-1}u\\ 0\end{pmatrix}
= \begin{pmatrix}(I-\widetilde Z Z)^{-1}u\\ Z(I-\widetilde Z Z)^{-1}u\end{pmatrix}.
\]
This gives \begin{equation}\label{QJZ}
	Q_{J(Z)}\Big|_{H^{1/2}_{+}}
	=\begin{pmatrix}(I-\widetilde{Z}Z)^{-1}\\
		Z(I-\widetilde{Z}Z)^{-1}
	\end{pmatrix}.
\end{equation}
By definition of \(\mathcal{Z}(J)\),
\[
\mathcal{Z}(J(Z)) 
= \big(P_-Q_{J(Z)}|_{H_+}\big)\big(P_+Q_{J(Z)}|_{H_+}\big)^{-1}
= Z(I-\widetilde Z Z)^{-1}\cdot (I-\widetilde Z Z) = Z,
\]
which establishes \eqref{eq:inverse-graph} and completes the proof.

\end{proof}

\begin{proposition}\label{prop:JhJZ}
For every $[h]\in\T$,  we have $  J_h=J(Z_h)$ and
 $[h]\mapsto J_h$ is operator-norm real analytic.
\end{proposition}

\begin{proof}
	Recall that the $\pm i$ eigenspaces of $J_h$ are given by
	\[
	E_{-i}(J_h)=V_h\bigl(H^{1/2}_+\bigr),\qquad
	E_{+i}(J_h)=V_h\bigl(H^{1/2}_-\bigr).
	\]
	From \eqref{eq:graph} we know $V_h(H^{1/2}_+)=\operatorname{Graph}(Z_h)$.
We compute
	\begin{align*}
		V_h(H_-^{1/2})
		&=V_h\big(\overline{H_+^{1/2}}\big)
		=\overline{V_h(H_+^{1/2})}
		=\overline{\operatorname{Graph}(Z_h)}\\
		&=\overline{\,\{u+Z_h u\;:\;u\in H_+\,\}}
		=\{\,\overline{u}+\overline{Z_h u}\;:\;u\in H_+\,\}\\
		&=\{\,v+\widetilde Z_h v\;:\;v\in H_-\,\}
		=\operatorname{Graph}(\widetilde Z_h).
	\end{align*}
Therefore
\[
E_{-i}(J_h)=\operatorname{Graph}(Z_h),\qquad
E_{+i}(J_h)=\operatorname{Graph}(\widetilde{Z_h}).
\]
Recall that $J(Z_h)=R_{Z_h}HR_{Z_h}^{-1}$ satisfies
\[
E_{-i}(J(Z_h))=\operatorname{Graph}(Z_h),\qquad
E_{+i}(J(Z_h))=\operatorname{Graph}(\widetilde Z_h).
\]
Thus $J_h$ and $J(Z_h)$ share the same $\pm i$ eigenspaces.
Hence, $J_h=J(Z_h)$.
The real analyticity of $[h]\mapsto J_h$ follows by composition
from Corollary~\ref{cor:Zholomorphic} and Proposition~\ref{prop:graph-J}.
\end{proof}

\subsection{Reconstruction from the complex-structure coordinate}
Recall that
\[
 \mathcal J=\{J_h:[h]\in\T\}\subseteq\Bop(H^{\frac{1}{2}}(S^{1})).
\]
Consider the open set
\begin{equation*}
 \mathcal U=\left\{J:
 P_+Q_J|_{H^{1/2}_{+}}\text{is \ invertible} \ \text{and} \
 E\bigl(F_-^{-1}\mathcal{Z}(J)F_+\bigr)\in\Omega=\beta(T(1))
\right\}.
\end{equation*}
Then $J_h=J(Z_h)$ belongs to $\mathcal U$. Indeed,
Propositions~\ref{prop:JhJZ} and~\eqref{QJZ} give
\[
 P_+Q_{J_h}|_{H^{1/2}_+}
 =\bigl(I-\widetilde{Z_h}Z_h\bigr)^{-1},
\]
which is invertible, while Propositions~\ref{prop:graph-grunsky}
and~\ref{prop:recovery} give
$E(F_-^{-1}Z(J_h)F_+)=\beta([h])\in\Omega$.

\begin{theorem}[Holomorphic reconstruction]\label{thm:local-reconstruction}
The formula
\begin{equation}\label{eq:Klocal}
 K(J)=\beta^{-1}\!\left(E\bigl(F_-^{-1}\mathcal{Z}(J)F_+\bigr)\right),
 \qquad J\in\mathcal U,
\end{equation}
defines a holomorphic map $K:\mathcal U\to\T$ satisfying
\begin{equation}\label{eq:Rperiod}
 K(J_h)=[h]\qquad([h]\in\T).
\end{equation}
Consequently, $\Phi([h])=J_h$ is a real-analytic diffeomorphism from $\T$
onto $\mathcal J$ in the operator-norm topology. Moreover,
\begin{equation}\label{eq:retraction}
 r=\Phi\circ K:\mathcal U\longrightarrow\mathcal J
\end{equation}
is an operator-norm real-analytic retraction of  $\mathcal U$ onto $\mathcal J$.
\end{theorem}

\begin{proof}
The set $\mathcal U$ is open on $B(H^{1/2})$. Proposition~\ref{prop:graph-J} shows that
$J\mapsto \mathcal{Z}(J)$ is holomorphic there. The Fourier maps and $E$ are bounded
complex linear, and $\beta^{-1}$ is holomorphic on $\Omega$; hence $K$ is
holomorphic.

For a  point $J_h$, Proposition~\ref{prop:graph-J} gives
$\mathcal{Z}(J_h)=Z_h$, and Proposition~\ref{prop:graph-grunsky} gives
$F_-^{-1}Z_hF_+=G([h])$. Proposition~\ref{prop:recovery} then yields
\[
 E\bigl(F_-^{-1}\mathcal{Z}(J_h)F_+\bigr)=\beta([h]),
\]
so \eqref{eq:Rperiod} follows.

The forward map is real analytic by Proposition~\ref{prop:JhJZ}, and
\eqref{eq:Rperiod} says that $K\circ\Phi=\id_{\T}$. Hence $\Phi$ is a
homeomorphism onto $\mathcal J$ with real-analytic inverse. Differentiating
$K\circ\Phi=\id_{\T}$ shows that $d\Phi$ has the bounded left inverse $dK$;
thus $\Phi$ is a split real-analytic immersion and hence a real-analytic
embedding. Finally, $r$ is real analytic, takes values in $\mathcal J$, and
fixes every point of $\mathcal J$. Since $\mathcal J\subset\mathcal U$, it follows that
$r^2=r$ and $r(\mathcal U)=\mathcal J$.
\end{proof}

\begin{proof}[Proof of Corollary~\ref{cor:topology}]
The equivalence \eqref{eq:top-equivalence} follows from continuity of the
forward map and the local inverse in
Theorem~\ref{thm:local-reconstruction}. For
\eqref{eq:localcomparison}, use the Bers coordinate near $[h_0]$ and shrink
to convex balls on which the derivatives of the maps
\[
 \phi\longmapsto J_{\beta^{-1}(\phi)}
 \quad\text{and}\quad
 J\longmapsto E\bigl(F_-^{-1}\mathcal{Z}(J)F_+\bigr)
\]
are bounded. The Banach-space mean-value estimate then makes both maps
Lipschitz on those smaller balls. The second map agrees with
$\beta\circ\Phi^{-1}$ on $\mathcal J$, which gives the left-hand inequality;
the first gives the right-hand inequality.
\end{proof}

\begin{proof}[Proof of Corollary~\ref{cor:hushen}]
 Since $A_h$ is invertible and $Z_h=B_hA_h^{-1}$, compactness and the
Hilbert--Schmidt property are equivalent for $B_h$ and $Z_h$.

By \eqref{eq:Jblock} and Proposition~\ref{prop:JhJZ},
\[
J_h-H=
\begin{pmatrix}
	-2i\,\widetilde{Z}_h Z_h\,(I-\widetilde{Z}_h Z_h)^{-1}
	& 2i\,\widetilde{Z}_h\,(I-Z_h\widetilde{Z}_h)^{-1}\\[4pt]
	-2i\,Z_h\,(I-\widetilde{Z}_h Z_h)^{-1}
	& 2i\,Z_h\widetilde{Z}_h\,(I-Z_h\widetilde{Z}_h)^{-1}
\end{pmatrix}.
\]
Since $(I-\widetilde{Z}_h Z_h)^{-1}$ and $(I-Z_h\widetilde{Z}_h)^{-1}$ are bounded operators,
compactness  and  Hilbert--Schmidt property of $Z_h$ clearly implies the same property for every block, hence for $J_h-H$. 
Conversely, the same holds.
\end{proof}

\section{ The differential at the identity}\label{sec:infinitesimal}
Nag and Sullivan
developed the infinitesimal period map in Beltrami coordinates, while the
infinitesimal Grunsky map and its Hilbert--Schmidt relation to the Bers
differential were computed in  \cite{ShenGrunsky}. We express the same
differential in
boundary-vector-field coordinates, where it becomes a Calder\'on commutator
and a weighted Hankel operator.

Let $h_t$ be a smooth  path  in
$\operatorname{Diff}_+(S^1)$,  the group of orientation-preserving
$C^\infty$ diffeomorphisms of $S^1$, with $h_0=\id$. Choose angular
lifts $\varphi_t$ satisfying
\[
h_t(e^{i\theta})=e^{i\varphi_t(\theta)},\qquad
\varphi_t(\theta+2\pi)=\varphi_t(\theta)+2\pi,
\]
and suppose that
\begin{equation}\label{eq:path}
	\varphi_t(\theta)=\theta+tA(\theta)+o(t)
	\quad\text{in }C^{\infty}(S^{1}),
\end{equation}
where $A\in C^\infty_{\mathbb{R}}(S^1)$. 
Only the class of $A$ modulo
$\operatorname{span}_{\R}\{1,\cos\theta,\sin\theta\}$ represents a tangent
vector to $\T$; we retain all three M\"obius modes until the kernel is
identified below. 

A smooth path $h_t\in\operatorname{Diff}_+(S^1)$ with initial vector field $A$ descends via the quotient projection
to a smooth curve in the Fr\'echet manifold $M=\operatorname{Diff}_+(S^1)/\operatorname{M\"ob}$.
By \cite[Theorem~1.1]{NagVerjovsky}, the holomorphic immersion $M\hookrightarrow T(1)$
carries this curve to a smooth real‑parameter curve inside the Teichm\"uller Banach manifold $T(1)$.
Corollary~\ref{cor:Zholomorphic} and Proposition~\ref{prop:JhJZ} therefore guarantee that
the operator‑norm derivatives of $t\mapsto J_{h_t}$ and $t\mapsto Z_{h_t}$ exist.
We set
\[
\dot J_A=\left.\frac{dJ_{h_t}}{dt}\right|_{t=0}
\in\Bop(H^{\frac12}),
\qquad
\dot Z_A=\left.\frac{dZ_{h_t}}{dt}\right|_{t=0}
\in\Bop(H^{\frac12}_+,H^{\frac12}_-).
\]
These derivatives depend only on the initial vector field $A$, and not on the particular path $h_t$ realizing $A$,
which justifies the notation.
The operator $\dot J_A$ preserves $H^{\frac12}_{\mathbb R}$; its restriction to that
real subspace is the differential of the compatible‑complex‑structure map $T(1)\to\mathcal B(H^{\frac12})$.

\subsection{The complex-structure derivative}

For $f\in C^\infty(S^1)$, define
\[
 (M_A\partial_\theta)f=Af'.
\]
Thus $M_A\partial_\theta$ is a well-defined densely defined operator on
$H^{\frac12}$. For each fixed smooth representative $f$, write
\begin{equation}\label{eq:Vdot}
 \dot V_Af:=\left.\frac d{dt}\right|_{t=0}V_{h_t}f
 =A f'=M_A\partial_\theta f,
\end{equation}
where equality is understood modulo constants; the convergence holds in
$C^\infty$, hence in $H^{\frac12}$. This is a strong derivative on the fixed
vector $f$, not an operator-norm derivative of
$t\mapsto V_{h_t}$; see Theorem~\ref{thm:separation} or
\eqref{eq:inverse-minimum-bound}.
A direct computation gives that the inverse path has infinitesimal vector field $-A$, so
\[
 \left.\frac d{dt}\right|_{t=0}V_{h_t}^{-1}f=-\dot V_Af.
\]

\begin{theorem}[Calder\'on commutator form]\label{thm:commutator}
The operator-norm derivative $\dot J_A$ is the bounded extension to
$H^{\frac{1}{2}}(S^{1})$ of the operator initially defined on
$C^\infty(S^1)$ by
\begin{equation}\label{eq:commutator}
 \dot J_A=[M_A\partial_\theta,H].
\end{equation}
\end{theorem}

\begin{proof}
	The operator‑norm derivative $\dot J_A = \bigl.\frac{d}{dt}\bigr|_{t=0}J_{h_t}$ exists
	by the preceding discussion.
	We now derive its explicit form on the dense subspace
	$C^\infty(S^1;\mathbb C)/\mathbb C$.
	
	For the smooth path $h_t\in\operatorname{Diff}_+(S^1)$ with $h_0=\mathrm{id}$
	and infinitesimal generator $A$, the pullback operator $V_{h_t}\colon f\mapsto f\circ h_t$
	has the Fréchet‑topology asymptotic expansion on $C^\infty(S^1)$:
	\[
	V_{h_t}=I+tM_A\partial_\theta+o(t),\qquad
	V_{h_t}^{-1}=I-tM_A\partial_\theta+o(t).
	\]
	Here $o(t)$ is understood in the Fréchet topology of $C^\infty(S^1)$,
	\textbf{not} in the operator‑norm topology of $\mathcal B(H^{\frac12})$.
	
	Substitute these expansions into $J_{h_t} = V_{h_t} H V_{h_t}^{-1}$ gives
	\begin{align*}
		J_{h_t} f
		&=\bigl(I+tM_A\partial_\theta+o(t)\bigr)H
		\bigl(I-tM_A\partial_\theta+o(t)\bigr)f\\
		&= Hf + t\bigl(M_A\partial_\theta H - H M_A\partial_\theta\bigr)f + o(t).
	\end{align*}

 For each $f\in C^\infty(S^1)$,
 $\frac{1}{t}(J_{h_t}-J_{h_0})$ converges in $H^{\frac12}$ to $\bigl[M_A\partial_\theta,H\bigr]f$.
On the other hand,  $\frac{1}{t}(J_{h_t}-J_{h_0})$ converges
in operator norm to $\dot J_A$, hence $\frac{1}{t}(J_{h_t}-J_{h_0})f\to \dot J_A f$.
Uniqueness of limits in $H^{\frac12}$ gives
\[
\dot J_A f = \bigl[M_A\partial_\theta,H\bigr]f
\quad \text{for all } f\in C^\infty(S^1).
\]
	The statement that $\dot J_A$ is the bounded extension of the commutator
	to $H^{\frac12}(S^1)$ now follows from density.  
\end{proof}

\subsection{The graph derivative}

\begin{theorem}\label{thm:graphderivative}
For $f\in H^{\frac{1}{2}}_{+}(S^{1})\cap C^{\infty}(S^{1})$, $ \dot Z_Af=P_-(Af').$
Consequently,
\[
 \dot Z_A=P_-M_A\partial_\theta|_{H^{1/2}_{+}(S^{1})}
\]
as a bounded operator from $H^{\frac{1}{2}}_{+}(S^{1})$ to
$H^{\frac{1}{2}}_{-}(S^{1})$.
\end{theorem}

\begin{proof}
Fix a smooth $f\in H^{\frac{1}{2}}_{+}(S^{1})$. Put
\[
 u_t=A_{h_t}f=P_+V_{h_t}f.
\]
Then $u_t\to f$ in $H^{\frac{1}{2}}_{+}(S^{1})$ and, by the definition of $Z_t$,
\[
 Z_tu_t=B_{h_t}f=P_-V_{h_t}f.
\]
Since $Z_0=0$ and $Z_t/t\to\dot Z_A$ in operator norm, we  have
\[\frac1t Z_t u_t
 =\left(\frac{Z_t}{t}-\dot Z_A\right)u_t+\dot Z_A\,(u_t-f)+\dot Z_A f\longrightarrow \dot Z_A f.
\]
On the other hand, $P_-f=0$, and \eqref{eq:Vdot} gives
\[\frac1t Z_t u_t=\frac1tP_-V_{h_t}f=
 \frac1tP_-\bigl(V_{h_t}f-f\bigr)
 \longrightarrow P_-(Af').
\]
This proves $ \dot Z_Af=P_-(Af')$ on a dense subspace. The bounded
extension is the operator-norm derivative supplied by the preceding path
argument.
\end{proof}

\begin{corollary}\label{cor:Ztilderformula}
Recall that $\widetilde{\dot Z}_A = C\dot Z_A C$, where $C$ denotes complex conjugation.
	Then
	\[
	\widetilde{\dot Z}_A = P_+ M_A\partial_\theta\big|_{H^{1/2}_{-}(S^1)}
	\]
	as a bounded operator from $H^{\frac12}_{-}(S^1)$ to $H^{\frac12}_{+}(S^1)$.
\end{corollary}
\begin{proof}
	This follows directly from Theorem~\ref{thm:graphderivative}. For $g\in H^{\frac{1}{2}}_{-}(S^{1})\cap C^{\infty}(S^{1}),$
	\begin{align*}
		\widetilde{\dot Z}_A g
		&= C P_- M_A \partial_\theta C g 
		= P_+ C M_A C \partial_\theta g 
		 \\
		&= P_+ M_A \, C C \, \partial_\theta g={P_+ M_A \partial_\theta\, g}.
	\end{align*}
\end{proof}

The derivatives $\dot J_A$ and $\dot Z_A$ represent the same tangent vector
on $ T(1)$ in different coordinate systems. The precise relation
is given in the following proposition.

\begin{proposition}\label{prop:JdotZdot}
	Relative to the orthogonal decomposition
	$H^{\frac12}(S^1)=H^{\frac12}_+(S^1)\oplus H^{\frac12}_-(S^1)$,
	the operator $\dot J_A$ has the block‑matrix representation
	\begin{equation}
		\dot J_A
		=
		\begin{pmatrix}
			0 & 2i\,\widetilde{\dot Z}_A\\
			-2i\,\dot Z_A & 0
		\end{pmatrix}.
	\end{equation}
\end{proposition}

\begin{proof}
For $f\in H^{\frac{1}{2}}_{+}(S^{1})\cap C^\infty(S^1)$, $Hf=-if$. Write
\[
 M_A\partial_\theta f=P_+M_A\partial_\theta f+P_-M_A\partial_\theta f.
\]
Then by $ P_-=\tfrac12(I-iH)$, 
\begin{align*}
 [M_A\partial_\theta,H]f
 &=-iM_A\partial_\theta f-H(M_A\partial_\theta f)\\
 &=-2iP_-M_A\partial_\theta f\\
 &=-2i\dot Z_Af.
\end{align*}
Similarly, for $g\in H^{\frac12}_{-}(S^1)\cap C^\infty(S^1)$, using $Hg=ig$
and Corollary~\ref{cor:Ztilderformula}, we obtain
\[
[M_A\partial_\theta,H]g = 2iP_+ M_A \partial_\theta\, g=2i\,\widetilde{\dot Z}_A g.
\]
Since smooth vectors are dense in $H^{\frac12}_+$ and $H^{\frac12}_-$,
the block‑matrix representation extends by continuity to all of $H^{\frac12}(S^1)$.
\end{proof}

\subsection{The weighted Hankel matrix}

Recall that
\[
 e_m^+(\theta)=\frac{e^{im\theta}}{\sqrt m},\quad m\geq1,
 \qquad
 e_\ell^-(\theta)=\frac{e^{-i\ell\theta}}{\sqrt\ell},\quad \ell\geq1,
\]
are orthonormal bases of $H^{\frac{1}{2}}_{+}(S^{1})$ and
$H^{\frac{1}{2}}_{-}(S^{1})$, respectively.

\begin{proposition}[Fourier matrix of the tangent map]\label{prop:Hankel}
For $m,\ell\geq1$,
\begin{equation}\label{eq:Hankelmatrix}
(\dot Z_A)_{\ell,m}\triangleq 
 \langle\dot Z_Ae_m^+,e_\ell^-\rangle
 =i\sqrt{m\ell}\,\wh A(-(m+\ell)).
\end{equation}
Equivalently,
\[
 \dot Z_Ae_m^+
 =i\sum_{\ell\geq1}\sqrt{m\ell}\,
 \wh A(-(m+\ell))e_\ell^-.
\]
\end{proposition}

\begin{proof}
By Theorem \ref{thm:graphderivative},
\begin{align*}
	\dot Z_A e_m^+
	&= P_-\!\left(A(\theta)\frac{d}{d\theta}\frac{e^{im\theta}}{\sqrt{m}}\right)\\
	&= i\sqrt{m}\,P_-\Bigl(A(\theta)e^{im\theta}\Bigr)\\
	&= i\sqrt{m}\,P_-\sum_{n\in\mathbb Z}\widehat A(n)e^{i(n+m)\theta}\\
	&= i\sqrt{m}\sum_{\ell\ge 1}\widehat A(-(m+\ell))\,e^{-i\ell\theta}\\
	&= i\sum_{\ell\ge1}\sqrt{m\ell}\,\widehat A(-(m+\ell))\,e_\ell^-.
\end{align*}
\end{proof}

The dependence on $m+\ell$ makes this a Hankel matrix. The factor
$\sqrt{m\ell}$ comes from the orthonormal normalization in the critical
Sobolev space. 

The matrix formula has two useful consequences.

\begin{corollary}[The M\"obius kernel]\label{cor:Mobiuskernel}
Let $A\in C^{\infty}_{\R}(S^{1})$. Then
\[
 \dot Z_A=0
 \quad\Longleftrightarrow\quad
 A\in\operatorname{span}_{\R}\{1,\cos\theta,\sin\theta\}.
\]
Thus the kernel of the differential before normalization is exactly infinitesimal M\"obius transformations.
\end{corollary}

\begin{proof}
	If $\dot Z_A=0$, \eqref{eq:Hankelmatrix} yields $\wh A(-n)=0$ for all $n\ge2$.
	Since $A$ is real‑valued, $\wh A(n)=\overline{\wh A(-n)}$, so $\wh A(n)=0$ whenever $|n|\ge2$,
	thus $A\in\operatorname{span}_{\mathbb R}\{1,\cos\theta,\sin\theta\}$.
	
	Conversely, suppose $A$ has Fourier support only for $|n|\le 1$.
	Since $m+\ell\ge2$ for all $m,\ell\ge1$, we have $\widehat A(-(m+\ell))=0$.
	By \eqref{eq:Hankelmatrix}, $\dot Z_A=0$.
	These vector fields are exactly the infinitesimal M\"obius transformations on $S^1$.
\end{proof}

For an operator $T:H^{\frac{1}{2}}_{+}(S^{1})\to H^{\frac{1}{2}}_{-}(S^{1})$,
we define $T$ to be Hilbert‑Schmidt whenever the series
\[
\|T\|_{\HS}^2\triangleq \sum_{m=1}^{\infty}\|Te_m^+\|_{H^{\frac{1}{2}}_{-}}^2<\infty.
\]
Recall that $H^{3/2}(S^1)$ is defined as the space of functions $f$ on $S^1$ with Fourier coefficients $\widehat{f}(n)$ satisfying
\[
\sum_{n\in\mathbb Z}|n|^3\bigl|\widehat{f}(n)\bigr|^2<\infty.
\]
  It is a well‑known fact that the tangent space of the Weil–Petersson Teichmüller space at the base point can be identified with
 $H^{3/2}(S^1)$ modulo infinitesimal M\"obius transformations,
 see~\cite{NagVerjovsky,ShenTangWP,TakhtajanTeo}.

\begin{corollary}[Weil--Petersson form]
\label{cor:WPform}
For $A\in C^{\infty}_{\R}(S^{1})$,
\begin{equation}\label{eq:WPform}
 \|\dot Z_A\|_{\HS}^2
 =\frac16\sum_{n=2}^{\infty}(n^3-n)|\wh A(-n)|^2.
\end{equation}
Consequently, after quotienting out the M\"obius modes, the differential
extends by completion from smooth vector fields to the Weil--Petersson
tangent space $H^{3/2}(S^{1})$ with values in the Hilbert--Schmidt
operators. The right-hand side of \eqref{eq:WPform} is, up to a constant factor, the Weil--Petersson quadratic form at the base point.
Hence the differential of the map $Z_{h_t}$ identifies the completed
Weil--Petersson tangent space with its image in
$\mathrm{HS}(H^{\frac12}_+,H^{\frac12}_-)$.
\end{corollary}

\begin{proof}
Using the orthonormal bases above and \eqref{eq:Hankelmatrix},
\begin{align*}
 \|\dot Z_A\|_{\HS}^2
 &=\sum_{m,\ell\geq1}m\ell\,
   |\wh A(-(m+\ell))|^2\\
 &=\sum_{n=2}^{\infty}
   \left(\sum_{m=1}^{n-1}m(n-m)\right)|\wh A(-n)|^2\\
   &=\frac16\sum_{n=2}^{\infty}(n^3-n)|\wh A(-n)|^2
\end{align*}
The series $\sum_{n=2}^{\infty}(n^3-n)|\wh A(-n)|^2$ is the base‑point Weil‑Petersson quadratic form.   
Since $n^3-n$ is equivalent to $n^3$ for $n\geq2$, finiteness of the series is equivalent to $H^{3/2}$ regularity modulo the frequencies $-1,0,1$. Hence, the tangent map admits a continuous extension  from smooth real vector fields to the Weil–Petersson tangent space $H^{3/2}(S^1)$ taking values in $\mathrm{HS}(H^{1/2}_+,H^{1/2}_-)$. 
\end{proof}

\end{document}